\documentclass[12pt,letterpaper]{article}
\usepackage[left=0.9in,top=0.9in,right=0.9in,bottom=0.9in]{geometry}
\usepackage{amsfonts,amsmath,amsthm,amssymb,latexsym,mathrsfs,stmaryrd}
\usepackage{subcaption}
\usepackage{hyperref}

\theoremstyle{theorem}
\newtheorem{theorem}{Theorem}[section]
\newtheorem{corollary}[theorem]{Corollary}
\newtheorem{lemma}[theorem]{Lemma}
\newtheorem{proposition}[theorem]{Proposition}

\theoremstyle{definition}

\newtheorem{example}[theorem]{Example}

\theoremstyle{remark}
\newtheorem{remark}[theorem]{Remark}

\numberwithin{equation}{section}

\usepackage{amsfonts,amsmath,amsthm,amssymb,latexsym,mathrsfs,stmaryrd} \usepackage{mathtools}

\usepackage{tikz}
\usetikzlibrary{matrix,arrows,decorations.pathmorphing}
\usepackage{tikz-cd}
\tikzset{commutative diagrams/.cd}

\newcommand\Z{{\mathbb Z}}
\newcommand\N{{\mathbb N}}

\newcommand\Fp{{{\mathbb F}_p}}
\newcommand\Fq{{{\mathbb F}_q}}

\newcommand\Zp{{{\mathbb Z}_p}}
\def\O{\mathcal O}

\DeclareMathOperator{\im}{im}

\DeclareMathOperator{\Spec}{Spec}

\DeclareMathOperator{\GL}{GL}

\newcommand\subeq{\subseteq}

\newcommand\onto{\twoheadrightarrow}

\newcommand{\map}[1][]{{\xrightarrow{#1}}}

\DeclarePairedDelimiter{\abs}{\lvert}{\rvert}

\DeclarePairedDelimiter{\set}{\{}{\}}
\DeclarePairedDelimiter{\pairing}{\langle}{\rangle}
\DeclarePairedDelimiter{\parens}{\lparen}{\rparen}

\DeclarePairedDelimiter\floor{\lfloor}{\rfloor}

\newcommand{\tuparens}[1]{\textup{(}#1\textup{)}}
\newcommand{\tu}{\textup}

\usepackage{enumitem}
\setenumerate[1]{label=\textup{(\alph*)}}
\setenumerate[2]{label=\textup{(\roman*)}}

\DeclareMathOperator{\Cont}{Cont}
\DeclareMathOperator{\Hilb}{Hilb}
\DeclareMathOperator{\LT}{LT}
\DeclareMathOperator{\LM}{LM}
\DeclareMathOperator{\Span}{span}

\newcommand{\lex}{\mathrm{lex}}
\newcommand{\hlex}{\mathrm{hlex}}
\newcommand{\phlex}{\prec_\hlex}
\newcommand{\inv}{\mathrm{inv}}

\begin{document}
\title{Lattices in $\Fq[[T]]^d$ and spiral shifting operators}
\author{
Yifeng Huang\footnote{Dept.\ of Mathematics, University of Southern California. Email Address: {\tt yifenghu@usc.edu}}, Ruofan Jiang\footnote{Dept.\ of Mathematics, University of California, Berkeley. Email Address: {\tt ruofanjiang@berkeley.edu}}
}
\date{\today}

\maketitle

\begin{abstract}
We investigate the algebra and combinatorics of an analogue of the Hermite normal form that classifies finite-index submodules of $\Fq[[T]]^d$. We identity both normal forms as instances of Gr\"obner basis theory under different monomial orders, where the Hermite normal form corresponds to the lex order, and the new normal form the hlex order. We note that the hlex normal form recovers the Smith normal form, a feature not enjoyed by the Hermite normal form. We also identify the combinatorial structure underlying the cell decomposition induced by the hlex normal form, which appears to be of independent interest. Notably, the statistics tracking the cell dimensions is compatible, in a certain way, with a collection of $d$ ``spiral shifting operators'' on $\N^d$, which pairwise commute and collectively act freely and transitively. Using these operators, we give direct proofs of some new combinatorial identities obtained by translating the results of Solomon \cite{solomon1977zeta} and Petrogradsky \cite{petrogradsky2007multiple} in terms of the hlex normal form. 
\end{abstract}

\section{Introduction}
Counting algebraic objects such as modules and rings frequently leads to interesting zeta functions \cite{cohenlenstra1984heuristics, huang2023mutually,kmtb2015distribution}. Counting submodules of a given module is no exception. Consider the Solomon zeta function
\begin{equation}
\zeta_\Z(\Z^d;s)=\sum_{M} (\Z^d:M)^{-s},
\end{equation}
where the sum extends over all finite-index submodules (or lattices) of $\Z^d$, and $(\Z^d:M)$ denotes the index of $M$. In 1977, Solomon \cite{solomon1977zeta} proved a beautiful formula
\begin{equation}\label{eq:solomon_Z}
\zeta_\Z(\Z^d;s)=\zeta(s)\zeta(s-1)\dots \zeta(s-d+1)
\end{equation}
using the Hermite normal form; see also four other proofs in the monograph \cite{lubotzkysegal2003subgroup}. The motivation of this paper is an analogous formula over the power series ring $\Fq[[T]]$:
\begin{equation}\label{eq:count_submodule}
\sum_{M} t^{\dim_{\Fq} (\Fq[[T]]^d/M)}=\frac{1}{(1-t)(1-tq)\dots(1-tq^{d-1})},
\end{equation}
where the sum extends over all finite-index submodules of $\Fq[[T]]^d$. 

Solomon's proof applies to any Dedekind domain $R$ with a well-defined Dedekind zeta function $\zeta_R(s)$, giving the direct analog of \eqref{eq:solomon_Z}:
\begin{equation}\label{eq:solomon_dedekind}
\zeta_R(R^d;s)=\zeta_R(s)\zeta_R(s-1)\dots \zeta_R(s-d+1),
\end{equation}
and hence \eqref{eq:count_submodule} as a special case. The content of this paper concerns a different proof as well as several independently interesting combinatorial objects arising from it. In order to introduce our proof, we shall first explain Solomon's proof of \eqref{eq:count_submodule} in the context of Gr\"obner basis theory \cite{hironaka1964resolution}. The theory of the Hermite normal form states that every finite-index submodule $M$ of $\Fq[[T]]^d$ can be uniquely expressed as the column span of a matrix of the form
\begin{equation}\label{eq:hermite}
M=\im 
\begin{bmatrix}
T^{n_1} & • & • & • \\ 
a_{21}(T) & T^{n_2} & • & • \\ 
\vdots & \vdots & \ddots & • \\ 
a_{d1}(T) & a_{d2}(T) & \cdots & T^{n_d}
\end{bmatrix},
\end{equation}
where $n_i\geq 0$ and each $a_{ij}(T)$ is a polynomial in $T$ of degree strictly less than $n_i$. (We say the zero polynomial has degree $-1$.) Therfore, the set of finite-index submodules of $\Fq[[T]]^d$ is \emph{stratified} into strata indexed by $n_1,\dots,n_d\geq 0$. From it, the formula \eqref{eq:count_submodule} follows elementarily, as in \cite{solomon1977zeta}. Now, we denote by $u_1,\dots,u_d$ the standard basis vectors of $\Fq[[T]]^d$ and let $f_j=T^{n_j}u_j+\sum_{i=j+1}^d a_{ij}(T)u_i$. Then $M$ is the submodule generated by $f_i$. Moreover, the requirement for the Hermite normal form is equivalent to saying that $(f_1,\dots,f_d)$ is a reduced Gr\"obner basis with respect to the monomial order $\prec_{\lex}$, where $t^a u_i \prec_{\lex} t^b u_j$ if $i<j$ or ($i=j$ and $a<b$). 

In the rest of this paper, we consider a different monomial order $\prec_{\hlex}$, where $t^a u_i \prec_{\hlex} t^b u_j$ if $a<b$ or ($a=b$ and $i<j$).\footnote{
The order $\prec_{\hlex}$ is both a lexicographic order with $u_1\prec \dots \prec u_d\prec T$ and a homogeneous lexicographic order with $u_1\prec \dots \prec u_d$ (the position of $T$ does not matter). We choose to call it ``hlex'' for the flavor reason that the set of monomials under $\prec_{\hlex}$ is order-isomorphic to $\N$, while it is not the case in $\prec_{\lex}$. 
} We give an alternative proof of \eqref{eq:count_submodule} by classifying reduced Gr\"obner bases with respect to $\prec_{\hlex}$. This gives a different stratification of the set of finite-index submodules of $\Fq[[T]]^d$, also indexed by $n_1,\dots,n_d\geq 0$. We compute the cardinality of each stratum in terms of a statistics we denote by $W$ \eqref{eq:def_weight_first}. The left-hand side of \eqref{eq:count_submodule} can now be expressed as a sum involving the $W$-statistics over all strata. To study this sum, we consider certain ``spiral shifting'' operators on the index set $X=\N^d=\set{(n_1,\dots,n_d):n_i\geq 0}$ of the strata. We prove combinatorial properties of these operators and how they interact with the $W$-statistics (Lemma \ref{lem:weight}), which leads to our proof of \eqref{eq:count_submodule}.

The hlex-stratification above can be interpreted as another normal form \eqref{eq:hlex_normal_form} which, just like the Hermite normal form, classifies square matrices over $\Fq[[T]]$ up to column operations. However, unlike the Hermite normal form, the new normal form recovers the Smith normal form by its diagonal entries (Proposition \ref{prop:hlex_recovers_smith}). As an application of this normal form, we give an alternative proof of a theorem of Petrogradsky \cite{petrogradsky2007multiple} that counts submodules $M$ of $\Fq[[T]]^d$ such that the quotient $\Fq[[T]]^d/M$ has a given $\Fq[[T]]$-module structure (a refinement of the information of cardinality). We do it by reducing the theorem to a refined sum involving the $W$-statistics. Our alternative proof also provides a direct explanation of a permutation sum in Petrogradsky's formula; see Remark \ref{rmk}. 

Apart from their applications to module counting, the spiral shifting operators also give a nontrivial self-bijection on $\N^d$ (Theorem \ref{thm:free_transitive}) that might lead to other combinatorial identities. We next state the combinatorial results independent of the context of submodule counting. 

\subsection{Main combinatorial results}
Fix $d\geq 1$. We denote by $\N$ the set of nonnegative integers, and we use the standard notation $[d]:=\set{1,2,\dots,d}$. Let $X=\N^d$ be the set of $d$-tuples $(n_1,\dots,n_d)$ of nonnegative integers. For $x=(n_1,\dots,n_d)\in X$, let $n(x)=n_1+\dots+n_d$ and let $W(x)\in \N$ be the statistics defined in Section \ref{sec:statistics}. The key identity we prove about the $n$- and $W$-statistics is the following result.

\begin{theorem}[Theorem \ref{thm:full_sum}]
As power series in $\Z[[t,q]]$, we have
\begin{equation}
\sum_{x\in X} t^{n(x)}q^{W(x)} = \frac{1}{(1-t)(1-tq)\dots(1-tq^{d-1})}.
\end{equation}

In particular, for given $n,W\geq 0$, the number of elements $x$ in $X$ with $n(x)=n$ and $W(x)=W$ equals the number of size-$W$ partitions whose Young diagram fits in an $n\times (d-1)$ rectangle.
\label{thm:full_sum_intro}
\end{theorem}

To prove Theorem \ref{thm:full_sum_intro}, we define self-maps $g_1,\dots,g_d$ on $X$ (Section \ref{sec:definition}). They satisfy the following properties. First, they commute.

\begin{theorem}[Theorem \ref{thm:commute}]
For all $1\leq j,j'\leq d$, we have
\begin{equation}
g_{j'}\circ g_j=g_j\circ g_{j'}
\end{equation}
as maps from $X$ to $X$.
\label{thm:commute_intro}
\end{theorem}

To state the next property, we note that as a result of Theorem \ref{thm:commute_intro}, the operators $g_j$ induce an action of the free abelian semigroup $\Gamma:=\N^d$ on $X$. For $a=(a_1,\dots,a_d)\in \Gamma$, the action is defined as
\begin{equation}\label{eq:def_action_intro}
a\cdot x:=g_1^{a_1}\circ \dots \circ g_d^{a_d}(x)
\end{equation}
for $x\in X$. 

The next result states that the action is free and transitive.

\begin{theorem}[Theorem \ref{thm:free_transitive}]
Let $\mathbf{0}=(0,\dots,0)\in X$. Then for any $x\in X$, there exists a unique $a\in \Gamma$ such that
\begin{equation}
a\cdot \mathbf{0} = x.
\end{equation}
\label{thm:free_transitive_intro}
\end{theorem}

The reason why the operators $g_j$ play a role in the proof of Theorem \ref{thm:full_sum_intro} is the following formula about how $g_j$ interacts with the $n$- and $W$-statistics. 

\begin{theorem}[Lemma \ref{lem:size} and Lemma \ref{lem:weight}]
For any $x\in X$ and $j\in [d]$, we have
\begin{align}
W(g_j(x))&=W(x)+j-1;\\
n(g_j(x))&=n(x)+1.
\end{align}
\label{lem:weight_intro}
\end{theorem}

The three properties above give a bijective proof of the equinumerosity statement in Theorem~\ref{thm:full_sum_intro}. We present the proof as it is instrumental in illustrating the main combinatorial idea of this paper.

\begin{corollary}
There is a bijection from the set of size-$W$ partitions with at most $n$ parts, each of which with size at most $d-1$, to the set of elements $x\in X$ with $n(x)=n$ and $W(x)=W$: 
\begin{equation}
[1]^{m_1}\dots [d-1]^{m_{d-1}} \mapsto g_1^{n-\sum_{i=1}^{d-1} m_i} g_2^{m_1} \dots g_d^{m_{d-1}}(\mathbf{0}).
\end{equation}
\end{corollary}
\begin{proof}
The map is well-defined (i.e., landing in the target) due to Theorem \ref{lem:weight_intro}. Its injectivity follows from the uniqueness part of Theorem \ref{thm:free_transitive_intro}. To show its surjectivity, assume $x\in X$ is such that $n(x)=n$ and $W(x)=W$. By the existence part of Theorem \ref{thm:free_transitive_intro}, we can write $x=g_1^{m_0}\dots g_d^{m_{d-1}}$ for some $m_0,\dots,m_{d-1}\geq 0$. It follows from Theorem \ref{lem:weight_intro} that $\sum_{i=0}^{d-1} m_i=n$ and $\sum_{i=1}^{d-1} im_i = W$, which implies that $[1]^{m_1}\dots [d-1]^{m_{d-1}}$ is in the source and is mapped to $x$. 
\end{proof}

Theorem \ref{thm:full_sum_intro} can be further generalized in terms of the semigroup action \eqref{eq:def_action_intro} defined by the operators $g_j$.

\begin{theorem}[Theorem \ref{thm:free_semigroup_orbit} and Corollary \ref{cor:semigroup_orbit}] \label{thm:semigroup_orbit_intro}
Let $\Gamma'$ be a finitely generated subsemigroup of $\Gamma$, and let $x_0$ be an element of $X$. Then $\sum_{x\in \Gamma'\cdot x_0} t^{n(x)}q^{W(x)}$ is a rational function in $t$ and $q$. Moreover, if $\Gamma'$ is free, then there is an explicit product formula in terms of any basis of $\Gamma'$. 
\end{theorem}

When $\Gamma'=\Gamma$ and $x_0=\mathbf{0}$, this result recovers Theorem \ref{thm:full_sum_intro}.

Our final result states that a refinement of the generating series in Theorem \ref{thm:full_sum_intro} also has rationality, which implies a theorem of Petrogradsky \cite{petrogradsky2007multiple}.

\begin{theorem}[Corollary \ref{cor:refined_sum} and Equation \eqref{eq:refined_sum_permutation}]
For $x=(n_1,\dots,n_d)$, let $\lambda_1(x)\geq \dots\geq\lambda_d(x)$ denote the sorting of $n_1,\dots,n_d$ in descending order. Then as power series in $q,t_1,\dots,t_d$, we have
\begin{equation}
\sum_{x\in X} q^{W(x)} t_1^{\lambda_1(x)}\dots t_d^{\lambda_d(x)}  
=\frac{\sum_{\pi\in S_d} \parens*{q^{-\inv(\pi)} \prod_{j\in D(\pi)} z_{j}}}{\prod_{j=1}^d (1-z_j)},
\end{equation}
where $z_j=q^{j(d-j)}t_1t_2\dots t_j$, $\inv(\pi)$ is the number of inversions of a permutation $\pi$, and $D(\pi)=\set{i\in [d-1]:\pi(i)>\pi(i+1)}$ is the set of descents of $\pi$.
\end{theorem}

\subsection{Further discussions}
Solomon's formula \eqref{eq:solomon_dedekind} can be put into the context of the moduli space of framed modules\footnote{See also \cite{henniguimaraes2021, nakajima1999hilbert} for the relation to framed quiver varieties.} and the punctual Quot scheme. Indeed, let $R$ be the coordinate ring of an affine variety over $\Fq$, and let $Y=\Spec R$. By passing to the quotient map, classifying finite-index submodules of $R^d$ amounts to classifying ``finite $d$-framed $R$-modules'', i.e., a finite $R$-module $N$ together with an $R$-linear surjection $R^d\onto N$ (``$d$-framing''). The moduli space of the latter classification problem is the Quot scheme of points parametrizing 0-dimensional quotients of $\O_Y^d$ \cite{grothendieck1995hilbert}. When $d=1$, the above Quot scheme is the Hilbert scheme of points on $Y$.

Solomon's formula concerns the enumerative (or motivic) aspect of the geometric objects above. If $Y=\Spec R$ is an affine smooth curve over $\Fq$, then \eqref{eq:solomon_dedekind} for $R$ follows from an analogous formula that computes of the motive of the punctual Quot scheme in the Grothendieck ring \cite{monavariricolfi2022}; the motive in the Grothendieck ring is a refinement of the point count. The same work \cite{monavariricolfi2022} also generlizes the formula to the nested Quot scheme. See also the work of Bifet \cite{bifet1989} on the punctual Quot scheme on a smooth curve and a generalization by Bagnarol, Fantechi and Perroni \cite{bfp2020motive} to the punctual Quot scheme of quotients of any vector bundle on a smooth curve.

There are a number of generating-series formulas about the moduli space of modules (with or without the framing) in the literature. For framed modules, see \cite{goettsche1990betti} on the Hilbert scheme of smooth surfaces and \cite{brv2020motivic, goettscheshende2014refined, maulikyun2013macdonald} on the Hilbert scheme of singular curves. For the moduli space of unframed modules, see formulas for smooth curves \cite{cohenlenstra1984heuristics}, smooth surfaces \cite{bryanmorrison2015motivic, feitfine1960pairs}, nodal singular curves \cite{huang2023mutually} and the (noncommutative) quantum plane \cite{huang2022quantum}.

In hindsight, some results of this paper can be understood from a higher level using affine Grassmannians. The space of submodules $M$ parametrized by hlex normal forms for a given $x=(n_1,\dots,n_d)$ as in \eqref{eq:hlex_normal_form} can be viewed as a Schubert cell in the affine Grassmannian $\GL_d(\Fq((T)))/\GL_d(\Fq[[T]])$ in the following precise sense:
\begin{equation}
    \set{M \text{ as in \eqref{eq:hlex_normal_form}}} = I_{-}\cdot  \mathrm{diag}(T^{n_1},\dots,T^{n_d}) \cdot \GL_d(\Fq[[T]])/\GL_d(\Fq[[T]]),
\end{equation}
where $I_{-}$ is the opposite Iwahori subgroup consisting of matrices in $\GL_d(\Fq[[T]])$ that are lower triangular modulo $T$; this can be verified by tracing the proof of Proposition~\ref{prop:hlex_recovers_smith}. This perspective satisfyingly explains why the hlex normal form recovers the Smith normal form: any $M$ as in \eqref{eq:hlex_normal_form} lives in the double coset $I_{-} \mathrm{diag}(T^{n_1},\dots,T^{n_d}) \GL_d(\Fq[[T]])$. The datum $(n_1,\dots,n_d)$ thus corresponds to an affine permutation and the $W$-statistics, which is the dimension of the corresponding Schubert cell, is the length of this affine permutation. From this perspective, our paper gives an explicit and computation-oriented description, via Gr\"obner basis and the spiral shifting operators, of the algebra and combinatorics of the Schubert cell decomposition of the affine Grassmannian.

The results of this paper easily extend to any DVR $(R,\pi,\kappa)$ in place of $\Fq[[T]]$, except the set of the form $\Span_{\Fq}\set{T^a,\dots,T^b}$ that shows up in \eqref{eq:hlex_normal_requirement} needs to be replaced by
\begin{equation}
    \pi^a K+\dots+\pi^bK,
\end{equation}
where $K$ is the image of any set-theoretical section of the quotient map $R\to \kappa$. (For example, if $(R,\pi,\kappa)=(\Zp,p,\Fp)$, then we may take $K=\set{0,\dots,p-1}$.) In brief, we may use \eqref{eq:hlex_normal_form} to parametrize all sublattices of $\Z_p^d$, as an alternative to Hermite normal form. Hermite parametrization is an important tool to count sublattices with additional requirements, such as forming a subring of $\Z_p^{\times d}$; see \cite{ckk2017cotype}. It is worth asking how the hlex normal form performs in these tasks.

\subsection{Organization}
In Section \ref{sec:groebner}, we give some preliminaries of Gr\"obner bases and describe our classification of finite-index submodules of $\Fq[[T]]^d$. In Section \ref{sec:definition}, we define the operators $g_j$. In Section \ref{sec:properties}, we prove several properties of $g_j$. In Section \ref{sec:statistics}, we define the $n$- and $W$-statistics and prove summation formulas about them using the operators $g_j$. Sections \ref{sec:definition}, \ref{sec:properties} and \ref{sec:statistics} are independent of Section \ref{sec:groebner}. In Section \ref{sec:smith}, we discuss the relation between our stratification to the Smith normal form and give an alternative proof of Petrogradsky \cite{petrogradsky2007multiple}.

\subsection*{Notation and conventions}
We fix $d\geq 1$ for the rest of the paper. We denote by $\N$ the set of nonnegative integers, and we use the standard notation $[d]:=\set{1,2,\dots,d}$. For $d\geq 1$, let $X=\N^d$ be the set of $d$-tuples $(n_1,\dots,n_d)$ of nonnegative integers.

\section{Gr\"obner bases}\label{sec:groebner}
Consider the free module $F:=\Fq[[T]]^d$ over $\Fq[[T]]$, and denote by $u_1,\dots,u_d$ the standard basis of $F$. A \textbf{monomial} is an element $T^n u_i$ for some $n\geq 0$ and $1\leq i\leq d$. We recall the monomial order $\phlex$ (denoted simply by $\prec$ from now on)
\begin{equation}
T^a u_i \prec T^b u_j\text{ iff }a<b\text{ or }(a=b\text{ and }i<j).
\end{equation}

To list the lowest few monomials,
\begin{equation}
u_1\prec u_2 \prec \dots \prec u_d \prec tu_1 \prec \dots tu_d \prec t^2 u_1 \prec \dots
\end{equation}

The monomials are indexed by the set $[d]\times \N$, via the identification
\begin{equation}
\mu_{\delta}:=T^n u_i\text{ if }\delta=(i,n)\in [d]\times \N.
\end{equation}

For any nonzero element $f=\sum_{\delta\in [d]\times \N} c_{\delta} \mu_{\delta}$ of $F$, the \textbf{leading term} and the \textbf{leading monomial} of $f$ are defined as
\begin{align}
\LT(f)&:=c_{\delta} \mu_{\delta},\\
\LM(f)&:=\mu_{\delta},
\end{align}
where $\mu_\delta$ is the lowest monomial such that $c_{\delta}$ is nonzero, which exists due to well-orderedness. For a submodule $M$ of $F$, the \textbf{leading submodule} of $M$ is defined as
\begin{equation}
\LT(M):=\set{\LT(f):f\in M}.
\end{equation}

We collect some standard facts about Gr\"obner bases over power series rings (or ``standard bases''), adapted to suit our purpose here; see \cite{becker1990stability} and \cite{hironaka1964resolution}. We use the noation $\pairing{f_1,\dots,f_r}$ to denote the $\Fq[[T]]$-submodule generated by $f_1,\dots,f_r$ in $F$.

\begin{proposition} Let $M$ be a finite-index submodule of $F$. Then
\begin{enumerate}
\item The leading submodule of $M$ can be uniquely expressed in the form
\begin{equation}
\LT(M)=\pairing{T^{n_1}u_1,\dots,T^{n_d}u_d}
\end{equation}
where $n_1,\dots,n_d\geq 0$. Moreover, there is an isomorphism of $\Fq$-vector spaces
\begin{equation}
\frac{F}{M}\cong \frac{F}{\LT(M)}.
\end{equation}

\item There is a unique generating set \tuparens{called the \textbf{reduced Gr\"obner basis} for $M$}
\begin{equation}
G=\set{f_1,\dots,f_d}
\end{equation}
for $M$ such that for every $i\in [d]$, we have
\begin{enumerate}
\item $\LT(f_i)=T^{n_i}u_i$.
\item Every nonleading term of $f_i$ is not divisible by $\LT(f_j)$ for any $j\in [d]$.
\end{enumerate}

\item Every tuple $\set{f_1,\dots,f_d}$ satisfying conditions \tu{(i)} and \tu{(ii)} in \tu{(b)} is the reduced Gr\"obner basis for the submodule $\pairing{f_1,\dots,f_d}$. 
\end{enumerate}
\label{prop:groebner}
\end{proposition}

\begin{remark}
The statement (c) is not obvious; it follows from the Buchberger criterion. A priori, not every tuple $(f_1,\dots,f_d)$ satisfying (i)(ii) is a reduced Gr\"obner basis, because the leading submodule of $\pairing{f_1,\dots,f_d}$ may be bigger than $\pairing{\LT(f_1),\dots,\LT(f_d)}$. The criterion to test whether $(f_1,\dots,f_d)$ is indeed a reduced Gr\"obner basis is the Buchberger criterion. In our case, for each basis element $u_i$, there is only one monomial (namely, $\LT(f_i)$) involving $u_i$, so the Buchberger criterion automatically passes.
\end{remark}

With Proposition \ref{prop:groebner}, we can now classify finite-index submodules of $\Fq[[T]]^d$ by dividing them into strata indexed by $x=(n_1,\dots,n_d)\in X$. Each stratum, denoted by $\Hilb(x)$, is the set of all finite-index submodules $M$ of $F$ such that
\begin{equation}
\LT(M)=\pairing{T^{n_1}u_1,\dots,T^{n_d}u_d}.
\end{equation}

For any $M$ in $\Hilb(x)$, an $\Fq$-basis for $F/\LT(M)$ is
\begin{equation}
\set{T^{a}u_i: 0\leq a<n_i, i\in [d]},
\end{equation}
so by Proposition \ref{prop:groebner}(a), 
\begin{equation}
\dim_{\Fq} F/M=\dim_{\Fq} F/\LT(M)=n_1+\dots+n_d.
\end{equation}

By the conditions (b)(i)(ii) in Proposition \ref{prop:groebner}, a typical element of $\Hilb(x)$ is given by the reduced Gr\"obner basis $(f_1,\dots,f_d)$, where
\begin{equation}\label{eq:reduced_basis}
f_i = T^{n_i}u_i + \text{(linear combination of $T^{a}u_j$ such that $a<n_j$ and $T^{a}u_j\succ T^{n_i}u_i$)}
\end{equation}

Given $i\neq j\in [d]$, we notice that in order for $f_i$ to have a term involving $T^a u_j$ for some $a$, we must have $T^{n_j}u_j \succ T^{n_i} u_i$, namely, $n_j+j/d > n_i + i/d$. Moreover, in this case, the number of monomials of the form $T^a u_j$ that could appear in $f_i$ is
\begin{equation}
\floor{n_j+\frac{j}{d} - n_i - \frac{i}{d}}.
\end{equation}

Since the coefficient of each monomial that could appear in the nonleading terms of $g_i$ can be chosen freely and independently, it follows from the discussion above that the cardinality of $\Hilb(x)$ is given by $q^{W(x)}$, where
\begin{equation}\label{eq:def_weight_first}
W(x)=\sum_{i\neq j\in [d]} \max\set{0,\floor{n_j+\frac{j}{d} - n_i - \frac{i}{d}}}.
\end{equation}

Define $n(x)=n_1+\dots+n_d$ for $x=(n_1,\dots,n_d)\in X$. We have just proved the following combinatorial formula for the left-hand side of \eqref{eq:count_submodule}.

\begin{lemma}
Using the definition of $W(x)$ and $n(x)$ above, we have
\begin{equation}
\sum_{M} t^{\dim_{\Fq} (\Fq[[T]]^d/M)} = \sum_{x\in X} q^{W(x)} t^{n(x)},
\end{equation}
where the sum on the left-hand side extends over all finite-index submodules of $\Fq[[T]]^d$.
\end{lemma}

Together with Theorem \ref{thm:full_sum_intro}, the formula \eqref{eq:count_submodule} follows. 

\section{Spiral shifting operators}
\label{sec:definition}
Recall that $X=\N^d$ is the set of $d$-tuples $(n_1,\dots,n_d)$ of nonnegative integers, and the goal of this section is to define the operators $g_1,\dots,g_d: X\to X$. The operator $g_1$ is defined by
\begin{equation}\label{eq:g1_def}
 g_1(n_1,\dots,n_d)=(n_d+1,n_1,n_2,\dots,n_{d-1}).
\end{equation}

In order to define the other operators $g_2,\dots,g_d$, it is helpful to view an element $x=(n_1,\dots,n_d)\in X$ as a configuration of $d$ points
\begin{equation}
\Delta(x):=\set{(1,n_1),(2,n_2),\dots,(d,n_d)}\subeq [d]\times \N.
\end{equation}

For a point $\delta=(i,n)$ in $[d]\times \N$, we define the \textbf{seat}, the \textbf{level} and the \textbf{height} of $\delta$ to be
\begin{align}
i(\delta)&:=i,\\
n(\delta)&:=n,\\
h(\delta)&:=n+\frac{i}{d}.
\end{align}

We linearly order the set $[d]\times \N$ according to the height:
\begin{equation}
\delta_1 \prec \delta_2 \iff h(\delta_1) < h(\delta_2).
\end{equation}

For any $s\in \N$ and $\delta\in [d]\times \N$, we define 
\begin{equation}\label{eq:def_shift}
\delta + s/d
\end{equation}
to be the unique element of $[d]\times \N$ with height $h(\delta)+s/d$. 

In particular, $(i,n)+1/d=(i+1,n)$ if $1\leq i\leq d-1$, and $(d,n)+1/d=(1,n+1)$. Essentially, we are arranging the points of $[d]\times\N$ on an upright cylinder with a \emph{spiral} tracing the points in the height order, so the operator $+s/d$ means shifting upwards following the spiral by $s$ steps. Note also that $(i,n)+1=(i,n)+d/d=(i,n+1)$, hence the notation. 

Now the operator $g_1$ can be alternatively understood as shifting every point upwards by one step:
\begin{equation}
g_1(x):=x'\text{ such that }\Delta(x')=\set{\delta+1/d:\delta\in \Delta(x)}.
\end{equation}

To define the operators $g_2,\dots,g_d$, we first list the points of $\Delta(x)$ in the height order:
\begin{equation}
\Delta(x)=\set{\delta^1(x)\prec \delta^2(x)\prec \dots \prec \delta^d(x)}.
\end{equation}

As an informal definition of $g_j$, the operator $g_j$ keeps $\delta^1(x),\dots,\delta^{j-1}(x)$ fixed, and simultaneously shifts $\delta^j(x),\allowbreak\dots,\delta^d(x)$ upwards to the next available seat (namely, a seat not occupied by the fixed points). More formally, define $g_j(x)=x'$ such that $\Delta(x')=\set{\delta^k(x'): k\in [d]}$, where
\begin{equation}\label{eq:def_gj}
\delta^k(x')=
\begin{cases}
\delta^k(x), & k<j; \\
\delta^k(x)+s_k/d, & k\geq j,
\end{cases}
\end{equation}
where $s_k$ is the minimal positive integer such that
\begin{equation}
i(\delta^k + s_k/d) \in \set{i(\delta^j(x)),\dots,i(\delta^d(x))}.
\end{equation}

Implicit in the definition is an important observation that $g_j$ preserves the height ranking among the points, i.e., we do have $\delta^1(x')\prec \dots \prec \delta^d(x')$ if $\delta^k(x')$ is given by definition \eqref{eq:def_gj}. Indeed, when going from $x$ to $x'$, the lowest $j-1$ points are fixed, and the remaining points are raised in a way that preserves the relative order within. 

Note the subtlety that the definition of $g_j$ involves an interplay between the height-rankings and the seats of the points.

\begin{center}
\begin{figure}[h]
\centering
\begin{subfigure}[b]{0.48\textwidth}
\begin{tikzpicture}
\node (10) at (0,0) {$\cdot$};
\node [right of=10] (20) {$\cdot$};
\node [right of=20] (30) {$\cdot$};
\node [right of=30] (40) {$\cdot$};
\node [right of=40] (50) {$\bullet$};
\node [right of=50] (60) {$\cdot$};
\node [right of=60] (70) {$\cdot$};

\node [above of=10] (11) {$\cdot$};
\node [above of=20] (21) {$\cdot$};
\node [above of=30] (31) {$\bullet$};
\node [above of=40] (41) {$\cdot$};
\node [above of=50] (51) {$\cdot$};
\node [above of=60] (61) {$\cdot$};
\node [above of=70] (71) {$\bullet$};

\node [above of=11] (12) {$\bullet$};
\node [above of=21] (22) {$\cdot$};
\node [above of=31] (32) {$\cdot$};
\node [circle,draw=black] [above of=41] (42) {$\bullet$};
\node [above of=51] (52) {$\cdot$};
\node [above of=61] (62) {$\circ$};
\node [above of=71] (72) {$\cdot$};

\node [above of=12] (13) {$\cdot$};
\node [circle,draw=black] [above of=22] (23) {$\bullet$};
\node [above of=32] (33) {$\cdot$};
\node [above of=42] (43) {$\circ$};
\node [above of=52] (53) {$\cdot$};
\node [circle,draw=black] [above of=62] (63)  {$\bullet$};
\node [above of=72] (73) {$\cdot$};

\node [above of=13] (14) {$\cdot$};
\node [above of=23] (24) {$\circ$};
\node [above of=33] (34) {$\cdot$};
\node [above of=43] (44) {$\cdot$};
\node [above of=53] (54) {$\cdot$};
\node [above of=63] (64) {$\cdot$};
\node [above of=73] (74) {$\cdot$};

\node [right of=74] (84) {};
\node [left of=13] (03) {};

\node [left of=10, node distance=3ex] (y0) {$0$};
\node [left of=11, node distance=3ex] (y1) {$1$};
\node [left of=12, node distance=3ex] (y2) {$2$};
\node [left of=13, node distance=3ex] (y3) {$3$};
\node [left of=14, node distance=3ex] (y4) {$4$};

\node [below of=10, node distance=3ex] (x1) {$1$};
\node [below of=20, node distance=3ex] (x2) {$2$};
\node [below of=30, node distance=3ex] (x3) {$3$};
\node [below of=40, node distance=3ex] (x4) {$4$};
\node [below of=50, node distance=3ex] (x5) {$5$};
\node [below of=60, node distance=3ex] (x6) {$6$};
\node [below of=70, node distance=3ex] (x7) {$7$};

\draw [-to] (23) -- (43);
\draw [-to] (42) -- (62);
\draw [-to, shorten >= 5ex] (63) -- (84);
\draw [-to, shorten <= 5ex] (03) -- (24);
\path[->] (12) edge [loop above] (12);
\path[->] (31) edge [loop above] (31);
\path[->] (50) edge [loop above] (50);
\path[->] (71) edge [loop above] (71);

\draw [-, shorten <= -2ex, shorten >= -2ex] (20) -- (24);
\draw [-, shorten <= -2ex, shorten >= -2ex] (40) -- (44);
\draw [-, shorten <= -2ex, shorten >= -2ex] (60) -- (64);
\end{tikzpicture}
\end{subfigure}
\begin{subfigure}[b]{0.2\textwidth}
$\bullet$ unmoved

$\odot$ moved from

$\circ$ moved to

$|$ available seats
\end{subfigure}
\caption{$g_5(2,3,1,2,0,3,1)=(2,4,1,3,0,2,1)$}
\label{fig:spiral}
\end{figure}
\end{center}

\begin{example}\label{eg:intro}
We use an example with $d=7, x=(n_1,\dots,n_7)=(2,3,1,2,0,3,1)\in X=\N^d$ and $j=5$ to illustrate the definition of $g_j(x)$, see Figure \ref{fig:spiral}. Here, the points of $x$ are plotted in solid dots. Then, since $j=5$, the lowest $j-1=4$ points in the spiral order are fixed by $g_5$. Note that when there is a tie on the levels, the point at the left is considered lower, so $(1,2)$ rather than $(4,2)$ are fixed. The rest of the points are shifted to the right (skipping the occupied seats), and their destinations are indicated by the hollow dots. In addition, the shift must follow the spiral order: $(6,3)$ is shifted to $(2,4)$, with level increase by $1$. Finally, reading the new configuration consisting of fixed points and hollow dots, we get $g_5(x)=(2,4,1,3,0,2,1)$.

As an additional example, we leave it to the readers to verify the following commutative diagram:
\begin{equation}
\begin{tikzcd}
x=(0,2,1,0,1) \arrow[mapsto]{r}{g_3}\arrow[mapsto]{d}{g_2} & (0,2,2,0,1)\arrow[mapsto]{d}{g_2}\\
(0,2,2,1,0) \arrow[mapsto]{r}{g_3} & (0,2,2,2,0)
\end{tikzcd}
\end{equation}
\end{example}

\section{Properties of the operators}\label{sec:properties}
The first property of the operators $g_1,\dots,g_d$ is that they commute. 

\begin{theorem}\label{thm:commute}
For all $x\in X$ and $j'<j$, we have
\begin{equation}
g_{j'}(g_j(x))=g_j(g_{j'}(x)).
\end{equation}
\end{theorem}

\begin{proof}
Because both $g_{j'}$ and $g_j$ leave the lowest $j'-1$ points $\delta^1,\dots,\delta^{j'-1}$ intact, we may remove these points and consider the remaining points only. Effectively, we are assuming $j'=1$ without loss of generality.

Let $k=j-1$ and $l=d-k$. Say the lowest $k$ points of $\Delta(x)$ are
\begin{equation}
\set{\delta^1(x),\dots,\delta^k(x)} = \set{(p_1,m_1),\dots,(p_k,m_k)},
\end{equation}
where $p_1<\dots<p_k$. (That is, we sort these points by their seats.) Let us also name the highest $l$ points of $x$ as
\begin{equation}
\set{\delta^j(x),\dots,\delta^d(x)} = \set{(q_1,n_1),\dots,(q_l,n_l)},
\end{equation}
where $q_1<\dots<q_l$. We shall prove that $g_1(g_j(x))=g_j(g_1(x))$ by describing both sides explicitly. As a convenient convention, we define $(d+1,n)$ to be the point $(1,n+1)$ in $[d]\times \N$.

By definition, we have
\begin{equation}
\Delta(g_j(x)) = \set{(p_1,m_1),\dots,(p_k,m_k) ,(q_1,n_l+1),(q_2,n_1),\dots,(q_l,n_{l-1})}
\end{equation}
and
\begin{equation}
\Delta(g_1(x)) = \set{(p_1+1,m_1),\dots,(p_k+1,m_k),(q_1+1,n_1),\dots,(q_l+1,n_l)},
\end{equation}
where the convention $(d+1,n)=(1,n+1)$ applies to exactly one of points in the list for $\Delta(g_1(x))$. Applying $g_1$ to $\Delta(g_j(x))$, we have
\begin{equation}
\Delta(g_1(g_j(x))) = \set{(p_1+1,m_1),\dots,(p_k+1,m_k), (q_1+1,n_l+1),(q_2+1,n_1),\dots,(q_l+1,n_{l-1})}.
\end{equation}

Now we apply $g_j$ to $\Delta(g_1(x))$. Notice that $g_1$ does not change the height ranking within the points, because it raises the height of every point by exactly $1/d$. Thus $(p_1+1,m_1),\dots,(p_k+1,m_k)$ continue to be the lowest $k$ points of $\Delta(g_1(x))$, so $g_j$ will shift $(q_1+1,n_1),\dots,(q_l+1,n_l)$. We need to consider two cases:
\begin{enumerate}
\item If $q_l\neq d$, then $(q_1+1,n_1),\dots,(q_l+1,n_l)$ are their honest coordinates, so shifting by $g_j$ results in $(q_1+1,n_l+1),(q_2+1,n_1),\dots,(q_l+1,n_{l-1})$. 
\item If $q_l=d$, then the honest coordinates of $(q_1+1,n_1),\dots,(q_l+1,n_l)$ are $(1,n_l+1),(q_1+1,n_1),\dots,(q_{l-1}+1,n_{l-1})$, so shifting by $g_j$ results in $(1,n_{l-1}+1),(q_1+1,n_l+1),\dots,(q_{l-1}+1,n_{l-2})$. But note that $(1,n_{l-1}+1)$ can be represented by $(d+1,n_{l-1}) = (q_l+1,n_{l-1})$ by our convention. 
\end{enumerate}

We see that in both cases, we have
\begin{equation}
\Delta(g_j(g_1(x))) = \set{(p_1+1,m_1),\dots,(p_k+1,m_k),(q_1+1,n_l+1),(q_2+1,n_1),\dots,(q_l+1,n_{l-1})},
\end{equation}
mathcing the expression of $\Delta(g_1(g_j(x)))$. Hence, we have proved the desired equality $g_1(g_j(x))=g_j(g_1(x))$.
\end{proof}

For a multi-index $a=(a_1,a_2,\dots,a_d)$ in $\N^d$, define
\begin{equation}
g^a(x):=g_1^{a_1}g_2^{a_2}\dots g_d^{a_d}(x).
\end{equation}

By commutativity (Theorem \ref{thm:commute}), we have
\begin{equation}
g^{a+b} = g^a \circ g^b \text{ for }a,b\in \N^d.
\end{equation}

In other words, the free abelian semigroup $\Gamma:=\N^d$ acts on the set $X=\N^d$ via
\begin{equation}
a\cdot x = g^a(x).
\end{equation}

The second property of the operators $g_1,\dots,g_d$ is that the semigroup action defined above is free and transitive, in the following sense.

\begin{theorem}\label{thm:free_transitive}
Consider the configuration $\mathbf{0}=(0,\dots,0)$ in $X=\N^d$. Then for any configuration $x\in X$, there exists a unique $a\in \Gamma=\N^d$ such that
\begin{equation}
g^a(\mathbf{0})=x.
\end{equation}

Equivalently, the map $a\mapsto g^a(\mathbf{0})$ is a bijection from $\Gamma=\N^d$ to $X=\N^d$. 
\end{theorem}

In fact, the correct choice of $a$ can be figured out by playing a puzzle game that is very intuitive with the aid of Figure~\ref{fig:spiral}: start with $y=\mathbf{0}$. Apply $g_1$ the right number of times to $y$ until $\delta^1(y)=\delta^1(x)$. Then apply $g_2$ the right number of times to $y$ until $\delta^2$ is in place. Repeat the process until we are done. We now make a precise formal proof.

\begin{lemma}\label{lem:tight}
For $2\leq r\leq d$, say a configuration $x\in X$ to be \textbf{$r$-tight} if $\delta^r(x)$ is the lowest point in
\begin{equation}
\set{i(\delta^{r}(x)),\dots,i(\delta^d(x))}\times \N
\end{equation}
that is higher than $\delta^{r-1}(x)$. Then for any $1\leq j\leq r-1$, the configuration $g_j(x)$ is $r$-tight whenever $x$ is $r$-tight.
\end{lemma}

\begin{proof}[Proof of Lemma \tu{\ref{lem:tight}}]
Notice that the lowest $j-1$ points of $x$ are fixed by $g_j$ and are not involved in $\set{i(\delta^{r}(x)),\dots,i(\delta^d(x))}\times \N$. By removing the lowest $j-1$ points of $x$ and reindexing the remaining seats, we may assume $j=1$. Then $g_1$ spiral shifts every point up by one step (in the sense of definition \eqref{eq:def_shift} where $s=1$), so the $r$-tightness is preserved. 
\end{proof}

\begin{proof}[Proof of Theorem \tu{\ref{thm:free_transitive}}]
Because $g_i$'s commute with each other, we may choose to set up the equation $g^a(\mathbf{0})=x$ in the form
\begin{equation}
g_d^{a_d} \dots g_1^{a_1}(\mathbf{0}) = x.
\end{equation}

Look at the lowest point $\delta^1$. Since the only operator that moves $\delta^1$ is $g_1$, we must have
\begin{equation}
\delta^1(g_1^{a_1}(\mathbf{0}) )= \delta^1(x),
\end{equation}
which determines $a_1$ uniquely. In fact, we have
\begin{equation}
a_1/d=h(\delta^1(x)) - h(\delta^1(\mathbf{0})),
\end{equation}
since $g_1$ increases the height of every point by $1/d$.

Having determined $a_1,\dots,a_{r-1}$ ($2\leq r\leq d$), we now determine $a_r$ by looking at the $r$-th lowest point $\delta^r$. Let 
\begin{equation}
x'=g_{r-1}^{a_{r-1}} \dots g_1^{a_1}(\mathbf{0}),
\end{equation}
and consider the equation
\begin{equation}
g_d^{a_d} \dots g_r^{a_r}(x')=x.
\end{equation}

The only operator among $g_d,\dots,g_r$ that moves $\delta^r$ is $g_r$, so we must have
\begin{equation}
\delta^r(g_r^{a_r}(x'))=\delta^r(x).
\end{equation}

It remains to show that there is a unique integer $a_r\geq 0$ satisfying the equation above. The uniqueness is clear because every application of $g_r$ increases the height of $\delta^r$ strictly (see \eqref{eq:def_gj}). It suffices to prove the existence. The key observation is that $\delta^r(x)\succeq \delta^r(x')$. We now prove this observation.

Note that $\mathbf{0}$ is $r$-tight in the sense of Lemma \ref{lem:tight}. Since $x'=g_{r-1}^{a_{r-1}} \dots g_1^{a_1}(\mathbf{0})$, by repetitively applying Lemma \ref{lem:tight}, we have that $x'$ is also $r$-tight. 

By construction, $\delta^{j}(x')=\delta^{j}(x)$ for all $1\leq j\leq r-1$. Thus, the sets $\set{\delta^{r}(x'),\dots,\delta^d(x')}$ and $\set{\delta^{r}(x),\dots,\delta^d(x)}$ are the same (denoted by $K$), because they are both the complement of $\set{\delta^1(x),\dots,\delta^{r-1}(x)}$. Since $x'$ is $r$-tight, the point $\delta^r(x')$ is the lowest point of $K\times \N$ that is higher than $\delta^{r-1}(x')$. Since $\delta^r(x)\succ \delta^{r-1}(x)=\delta^{r-1}(x')$ and the seat of $\delta^r(x)$ also lies in $K$, we have $\delta^r(x)\succeq \delta^r(x')$, proving the desired observation.

Now, we recall from the definition \eqref{eq:def_gj} that for any configuration $y$ such that $\set{\delta^{r}(y),\dots,\allowbreak\delta^d(y)}\allowbreak=K$, the $r$-th lowest point $\delta^r(g_r(y))$ of $g_r(y)$ is defined to be the successor of $\delta^r(y)$ in $K\times \N$. Applying the above fact to $y=x', g_r(x'), g_r^2(x'),\dots$, we see that $\delta^r(g_r^{a_r}(x'))$ ranges over all points of $K\times \N$ that are higher than or equal to $\delta^r(x')$. Together with the observation $\delta^r(x)\succeq \delta^r(x')$ that we have just proved, it follows that there exists $a_r\geq 0$ such that $\delta^r(g_r^{a_r}(x'))=\delta^r(x)$.
\end{proof}

\begin{corollary}\label{cor:free_anywhere}
The group $\Gamma$ acts on any configuration $x\in X$ freely, namely, if $g^a(x)=g^b(x)$, then $a=b$.
\end{corollary}
\begin{proof}
By the existence part of Theorem \ref{thm:free_transitive}, there is $c\in \Gamma$ such that $x=g^c(\mathbf{0})$. Hence
\begin{equation}
g^{a+c}(x)=g^{b+c}(x).
\end{equation}

By the uniqueness part of Theorem \ref{thm:free_transitive}, we have $a+c=b+c$. Since the semigroup $\Gamma=\N^d$ is cancellative, we have $a=b$.
\end{proof}

\section{Some combinatorial statistics}\label{sec:statistics}
For a configuration $x=(n_1,\dots,n_d)$ in $X$, we define its \textbf{size} to be
\begin{equation}
n(x):=n_1+\dots+n_d.
\end{equation}

\begin{lemma}\label{lem:size}
For any $x\in X$ and $j\in [d]$, we have
\begin{equation}
n(g_j(x))=n(x)+1.
\end{equation}
\end{lemma}

\begin{proof}
When we apply $g_j$ to $x$, the rightmost point (i.e., the point with the largest seat) among $\delta^j(x),\dots,\delta^d(x)$ loops around and increases the level by 1, while the levels of other points are unchanged. 
\end{proof}

We now introduce a more interesting statistics. For two points $\delta_1 \prec \delta_2$ in $[d]\times \N$, we define their \textbf{distance} to be
\begin{equation}\label{eq:def_dist}
b(\delta_1,\delta_2):=\floor{h(\delta_2)-h(\delta_1)}.
\end{equation}

For a configuration $x\in X$ consisting of points $\Delta(x)=\set{\delta^1(x)\prec \dots \prec \delta^d(x)}$, we define the \textbf{weight} of $x$ to be
\begin{equation}\label{eq:def_weight}
W(x):=\sum_{j<k} b(\delta^j(x),\delta^k(x)).
\end{equation}

We note that this definition of $W(x)$ is equivalent to \eqref{eq:def_weight_first}. The statistics turns out to interact consistently with the operators $g_j$.

\begin{lemma}\label{lem:weight}
For any $x\in X$ and $j\in [d]$, we have
\begin{equation}
W(g_j(x))=W(x)+j-1.
\end{equation}
\end{lemma}

\begin{proof}
The main idea is intuitive once the reader attempts this lemma as an exercise aided by Figure \ref{fig:spiral}: each of the $j-1$ unmoved particle accounts for the increase of the $W$-statistics by $1$; for each unmoved particle, there is exactly one moved particle that skips its seat, and it is the distance between them that increases by $1$. 

For rigor, we record a precise proof. We call the points in
\begin{equation}
L:=\set{\delta^1(x),\dots,\delta^{j-1}(x)}
\end{equation}
low points, and the points in
\begin{equation}
H:=\set{\delta^j(x),\dots,\delta^d(x)}
\end{equation}
high points. The weight of $x$ consists of three parts: distances within $L$, distances within $H$, and distances from points of $L$ to points of $H$. When applying $g_j$, every distance within $L$ does not change (because every point of $L$ is fixed). We claim that every distance within $H$ does not change as well. To prove the claim, note that when defining distances between two points of $H$, we may ignore the seats of $L$ and reindex the seats of $H$ by $1,\dots,d-j+1$. Thus, we may assume $j=1$ without loss of generality. In this case, $g_1$ increases the height of every point by $1/d$, so the distance is unchanged from definition \eqref{eq:def_dist}.

Therefore, the weight difference $W(g_j(x))-W(x)$ is solely contributed by the distances from low points to high points. Let $\alpha\in L$ and $\beta\in H$. Then we note as an important observation that while $\beta$ is shifted upward following the spiral, the distance $b(\alpha,\beta)$ increases by 1 every time the seat of $\beta$ passes through the seat of $\alpha$. More formally, for any $\alpha\prec \beta\prec \beta'\in [d]\times \N$, we have
\begin{equation}
b(\alpha,\beta')-b(\alpha,\beta) = \#\parens[\bigg]{(\beta,\beta']\cap (\set{i(\alpha)}\times \N)},
\end{equation}
where $(\beta,\beta']\cap (\set{i(\alpha)}\times \N)$ is the set of all points $\delta$ with the same seat as $\alpha$ such that $\beta\prec \delta\preceq \beta'$. 

Sort the seats of high points in ascending order:
\begin{equation}
i(H)=\set{i_1<\dots<i_{d-j+1}}.
\end{equation}

Let $\beta_k$ denote the point of $H$ seated at $i_k$. Let $\beta'_k$ denote the result of $\beta_k$ after a shift by $g_j$. Then the seat of $\beta'_k$ is $i_{k+1}$ if $k<d-j+1$, and $i_1$ if $k=d-j+1$.

Partition the set of seats of points in $L$ by intervals
\begin{equation}
I_1=(i_1,i_2),\dots,I_{d-j}=(i_{d-j},i_{d-j+1}), I_{d-j+1}=(i_{d-j+1},d] \cup [1,i_1).
\end{equation}

Denote by $L_k$ the set of low points seated in $I_k$. Note that upon applying $g_j$, the point $\beta_k$ is ``shifted across'' the interval $I_k$. The discussions so far imply that
\begin{equation}
b(\alpha,\beta'_k)-b(\alpha,\beta_k)=
\begin{cases}
1, & \alpha\in I_k; \\
0, & \alpha\in L\setminus I_k.
\end{cases}
\end{equation}

It follows that
\begin{align}
W(g_j(x))-W(x) &= \sum_{\substack{\alpha\in L\\1\leq k\leq d-j+1}}  b(\alpha,\beta'_k)-b(\alpha,\beta_k) \\
&= \sum_{1\leq k\leq d-j+1} \abs{I_k} \\
&= \abs{L} = j-1,
\end{align}
as desired.
\end{proof}

We are now ready to prove Theorem \ref{thm:full_sum_intro}. For $x\in X$, we define the \textbf{content} of $x$ to be the monomial
\begin{equation}
\Cont(x):=t^{n(x)}q^{W(x)} \in \Z[[t,q]]
\end{equation}
in formal variables $t$ and $q$. 

\begin{theorem}\label{thm:full_sum}
We have an identity of power series in $\Z[[t,q]]$:
\begin{equation}
\sum_{x\in X} \Cont(x) = \frac{1}{(1-t)(1-tq)\dots(1-tq^{d-1})}.
\end{equation}

In particular, the number of configurations $x$ with $n(x)=n$ and $W(x)=W$ equals the number of size-$W$ partitions of length at most $n$ with parts at most $d-1$.
\end{theorem}

\begin{proof}
By Theorem \ref{thm:free_transitive}, we have
\begin{equation}
\sum_{x\in X} \Cont(x)=\sum_{a\in \Gamma} \Cont(g^a(\mathbf{0})).
\end{equation}

By Lemmas \ref{lem:size} and \ref{lem:weight} and noting that $\Cont(\mathbf{0})=1$, we have
\begin{align}
\sum_{x\in X} \Cont(x) &= \sum_{a_1,\dots,a_d\geq 0} g_1^{a_1}\dots g_d^{a_d}(\mathbf{0}) \\
&= \sum_{a_1,\dots,a_d\geq 0} t^{a_1+\dots+a_d} q^{0 a_1+1 a_2 + \dots + (d-1) a_d} \\
&= \frac{1}{(1-t)(1-tq)\dots(1-tq^{d-1})}.
\end{align}

The final assertion of the theorem follows from the elementary identity
\begin{equation}
\sum_{n=0}^\infty \sum_{\lambda\subeq [d-1]\times [n]} t^n q^{\abs{\lambda}} = \frac{1}{(1-t)(1-tq)\dots(1-tq^{d-1})},
\end{equation}
where the inner sum extends over all partitions of length at most $n$ with parts of size at most $d-1$, and $\abs{\lambda}$ denotes the size of $\lambda$.
\end{proof}


We note the following alternative proof of Theorem \ref{thm:full_sum} that uses only $g_1$.

\begin{proof}[Alternative proof of Theorem \tu{\ref{thm:full_sum}}]
Since we will vary $d$, we write $X$ as $X_d$. Observe from \eqref{eq:g1_def} that the map $g_1:X_d\to X_d$ is injective with image being the complement of
\begin{equation}
0\times X_{d-1} := \set{(0,n_2,\dots,n_d)\in X}.
\end{equation}

For any $x\in X_d$, we note that $n(g_1(x))=n(x)+1$ and $W(g_1(x))=W(x)$, the latter being because $g_1$ preserves the distance $b(\delta^j(x),\delta^k(x))$ for every $j<k$. 

We notice that for $x=(0,x')\in 0\times X_{d-1}$, we have $n(x)=n(x')$ and $W(x)=W(x')+n(x')$, the latter being because the distance from the point $(1,0)$ to $(i,n_i)$ is $n_i$. 

Let $f_d(t,q):=\sum_{x\in X_d} \Cont(x)$. Then we have a recurrence
\begin{align}
f_d(t,q) &= \sum_{x\in X_d} \Cont(g_1(x)) + \sum_{x' \in X_{d-1}} \Cont((0,x')) \\
&= \sum_{x\in X_d} t^{n(x)+1}q^{W(x)} + \sum_{x'\in X_{d-1}} t^{n(x')}q^{W(x')+n(x')}\\
&= tf_d(t,q)+f_{d-1}(tq,q).
\end{align}

Hence
\begin{equation}\label{eq:recurrence}
f_d(t,q)=\dfrac{1}{1-t} f_{d-1}(tq,q).
\end{equation}

Since it is vacuously true that $f_0(t,q)=1$, by applying \eqref{eq:recurrence} repetitvely, we have $f_1(t,q)=(1-t)^{-1}$, $f_2(t,q)=(1-t)^{-1} f_1(tq,q)=(1-t)^{-1}(1-tq)^{-1}$, $\dots$, $f_d(t,q)=(1-t)^{-1}(1-tq)^{-1}\dots (1-tq^{d-1})^{-1}$.
\end{proof}

Compared to the alternative proof, the \emph{current} proof of Theorem \ref{thm:full_sum} has the advantages that it suffices to work on a fixed $d$, and that the product form of the generating function is manifest. Furthermore, a small modification of this proof gives a summation formula over the forward orbit of any element in $X$ under any finitely generated subsemigroup of $\Gamma$.

\begin{theorem}\label{thm:free_semigroup_orbit}
Let $x_0$ be a configuration in $X$ and $\Gamma'$ be a subsemigroup of $\Gamma$ freely generated by $a^{(1)},\dots,a^{(r)}$. Then
\begin{equation}
\sum_{x\in \Gamma'\cdot x_0} \Cont(x) = \Cont(x_0)\prod_{i=1}^r \frac{1}{1-\Cont(a^{(i)})},
\end{equation}
where the content of an element $a=(a_1,\dots,a_d)$ of $\Gamma$ is defined as
\begin{equation}\label{eq:cont_gamma}
\Cont(a)=t^{a_1+\dots+a_d}q^{0a_1+1a_2+\dots+(d-1)a_d}.
\end{equation}
\end{theorem}

\begin{proof}
Corollary \ref{cor:free_anywhere} implies
\begin{equation}
\sum_{x\in \Gamma'\cdot x_0} \Cont(x) = \sum_{a\in \Gamma'} \Cont(g^a(x_0)).
\end{equation}

Note from Lemma \ref{lem:size} and Lemma \ref{lem:weight} that
\begin{equation}
\Cont(g^a(x))=\Cont(a)\Cont(x)
\end{equation}
for all $a\in \Gamma$ and $x\in X$. Note also that $\Cont(a+b)=\Cont(a)\Cont(b)$ for $a,b\in \Gamma$. 

Since elements of $\Gamma'$ can be traversed by the set of integers $k_1,\dots,k_r\geq 0$ by setting
\begin{equation}
a=k_1 a^{(1)}+\dots+k_r a^{(r)},
\end{equation}
we have
\begin{align}
\sum_{x\in \Gamma'\cdot x_0} \Cont(x) &= \sum_{k_1,\dots,k_r\geq 0} \Cont(g^{k_1 a^{(1)}+\dots+k_r a^{(r)}} x_0) \\
&= \sum_{k_1,\dots,k_r\geq 0} \Cont(k_1 a^{(1)}+\dots+k_r a^{(r)}) \Cont(x_0) \\
&= \Cont(x_0) \sum_{k_1,\dots,k_r\geq 0} \Cont(a^{(1)})^{k_1} \dots \Cont(a^{(r)})^{k_r} \\
&= \Cont(x_0)\prod_{i=1}^r \sum_{k_i\geq 0} \Cont(a^{(i)})^{k_i} \\
&= \Cont(x_0)\prod_{i=1}^r \frac{1}{1-\Cont(a^{(i)})}.
\end{align}
\end{proof}

\begin{corollary}\label{cor:semigroup_orbit}
Let $x_0$ be a configuration in $X$ and $\Gamma'$ be a finitely generated subsemigroup of $\Gamma$. Then $\sum_{x\in \Gamma'\cdot x_0} \Cont(x)$ is a rational function in $t$ and $q$. 
\end{corollary}
\begin{proof}
By a theorem of Ito \cite[Theorem 2]{ito1969semilinear}, any finitely generated subsemigroup of $\Gamma=\N^d$ is a disjoint union of finitely many cosets of (possibly trivial) free subsemigroups of $\Gamma$. Hence, there are free subsemigroups $\Gamma^1, \dots, \Gamma^r$ of $\Gamma$ and elements $a^1,\dots,a^r$ of $\Gamma$ such that
\begin{equation}
\Gamma'=\bigsqcup_{i=1}^r  a^i \Gamma^i.
\end{equation}

It follows that
\begin{equation}
\sum_{x\in \Gamma'\cdot x_0} \Cont(x) = \sum_{i=1}^r \sum_{x\in \Gamma^i\cdot (a^i x_0)} \Cont(x),
\end{equation}
where the right-hand side is a finite sum of rational functions by Theorem \ref{thm:free_semigroup_orbit}.
\end{proof}

\section{Relation to Smith normal form}\label{sec:smith}
We may translate the language of hlex Gr\"obner bases back to a matrix normal form that is different from the Hermite normal form \eqref{eq:hermite}. In particular, the general form of the reduced Gr\"obner basis described in Section \ref{sec:groebner} implies that every finite-index submodule $M$ of $\Fq[[T]]^d$ can be uniquely expressed as the column span of a matrix of the form
\begin{equation}\label{eq:hlex_normal_form}
M=\im 
\begin{bmatrix}
T^{n_1} & a_{12}(T) & \cdots & a_{1d}(T) \\ 
a_{21}(T) & T^{n_2} & \cdots & a_{2d}(T) \\ 
\vdots & \vdots & \ddots & \vdots \\ 
a_{d1}(T) & a_{d2}(T) & \cdots & T^{n_d}
\end{bmatrix},
\end{equation}
where \eqref{eq:reduced_basis} translates into
\begin{equation}\label{eq:hlex_normal_requirement}
a_{ij}(T) \in 
\begin{cases}
\Span \set{T^b: n_j<b<n_i}, & \text{if }i<j;\\
\Span \set{T^b: n_j\leq b<n_i}, & \text{if }i>j.
\end{cases}
\end{equation}

A special property of this new normal form is that the diagonal entries, with the order forgotten, recover the Smith normal form, or equivalently, the $\Fq[[T]]$-module structure of $\Fq[[T]]^d/M$. This is not the case for the Hermite normal form. For example, consider
\begin{equation}
A= \begin{bmatrix}
T & 0\\
1 & T
\end{bmatrix}.
\end{equation}
Then the Smith normal form of $A$ is $\mathrm{diag}(1,T^2)$ (up to permutation of diagonal entries), not $\mathrm{diag}(T,T)$. 

\begin{proposition}\label{prop:hlex_recovers_smith}
Let $M$ be of the form \eqref{eq:hlex_normal_form}. Then as an $\Fq[[T]]$-module, we have
\begin{equation}
\frac{\Fq[[T]]^d}{M}\cong \frac{\Fq[[T]]}{T^{n_1}}\oplus \dots \oplus \frac{\Fq[[T]]}{T^{n_d}}.
\end{equation}
\end{proposition}

\begin{proof}
Recall that to obtain the Smith normal form, we are allowed to perform row and column operations. Consider $x=(n_1,\dots,n_d)$ and let $i_1,\dots,i_d$ be the seats of $\delta^1(x),\dots,\delta^d(x)$. Look at the $i_1$-th row. Because $\delta^1(x)=(i_1,n_{i_1})$ is the lowest among all $\delta^j(x)$, we observe from \eqref{eq:hlex_normal_requirement} that the $i_1$-th row only has the diagonal entry. We also observe from \eqref{eq:hlex_normal_requirement} that every entry of a column is divisible by the diagonal entry of this column. Therefore, we may use the $i_1$-th row to eliminate the nondiagonal entries of the $i_1$-th column through row operations, without causing other changes. 

We continue the process inductively. Having cleared the nondiagonal entries of the $i_1$-th, $\dots$, $i_{k-1}$-th columns, we look at the $i_k$-th row. Because $\delta_k(x)$ is the lowest among $\delta_j(x)$, $j\geq k$, the entries $a_{i_k,i_j}(T)$ must be zero for $j>k$ for similar reasons. For the entries $a_{i_k,i_j}(T)$ for $j<k$, they are already cleared in the previous steps of the process. Hence the $i_k$-th row only has the diagonal entry, and we may proceed similarly.

After finishing the process at $k=d$, we obtain the Smith normal form $\mathrm{diag}(T^{n_1},\dots,T^{n_d})$ (or any permutation) and we are done.
\end{proof}

\begin{remark}
Both Proposition \ref{prop:groebner} and Proposition \ref{prop:hlex_recovers_smith} hold if we replace $\Fq$ by any field. We did not state them in the this generality to avoid distraction.
\end{remark}

In \cite{petrogradsky2007multiple}, Petrogradsky computed a refinement of the Solomon zeta function that not only remembers the index of a sublattice $M$ of $\Z^d$, but also the $\Z$-module structure of $\Z^d/M$. The essence of Petrogradsky's theorem is a local result that holds for all discrete valuation ring.

\begin{theorem}[Petrogradsky {\cite[Theorem 3.1(5), Theorem 8.1(2)]{petrogradsky2007multiple}}]
Let $R$ be a discrete valuation ring with uniformizer $\varpi$ and residue field $\Fq$. For a finite-index submodule $M$ of $R^d$, we define $\lambda_1(M),\dots,\lambda_d(M)$ by
\begin{equation}
\frac{R^d}{M}\cong \frac{R}{\varpi^{\lambda_1(M)}} \oplus \dots \oplus \frac{R}{\varpi^{\lambda_d(M)}}, \quad \lambda_1(M)\geq \dots\geq \lambda_d(M).
\end{equation}

Consider the series
\begin{equation}
Z(R^d;t_1,\dots,t_d):=\sum_{(R^d:M)<\infty} t_1^{\lambda_1(M)} \dots t_d^{\lambda_d(M)}
\end{equation}
in variables $t_1,\dots,t_d$.

Then $Z(R^d;t_1,\dots,t_d)$ is a rational function in $q,t_1,\dots,t_d$ that depends only on $d$, with denominator dividing $\prod_{j=1}^d (1-z_j)$, where $z_j=q^{j(d-j)}t_1t_2\dots t_j$. 
\label{thm:petrogradsky}
\end{theorem}

We note that $Z(R^d;t,t,\dots,t)$ recovers the Solomon zeta function for $R^d$. The rational formula for $Z(R^d;t_1,\dots,t_d)$ is also given in \cite{petrogradsky2007multiple}, restated below. For a permutation $\pi\in S_d$ of $[d]$, define the set of descents as $D(\pi)=\set{i\in [d-1]: \pi(i)>\pi(i+1)}$, and let $\inv(\pi):=\set{i<j: \pi(i)>\pi(j)}$ denote the number of inversions of $\pi$. Define
\begin{equation}\label{def:wDq}
w_D(q):=\sum_{\pi\in S_d, D(\pi)=D} q^{\inv(\pi)}.
\end{equation}

Then
\begin{equation}\label{eq:petrogradsky}
Z(R^d;t_1,\dots,t_d)=\frac{\sum_{D\in [d-1]} \parens*{w_D(q^{-1})\prod_{j\in D} z_j}}{\prod_{j=1}^d (1-z_j)},
\end{equation}
where $z_j=q^{j(d-j)}t_1t_2\dots t_j$. 

Proposition \ref{prop:hlex_recovers_smith} immediately expresses the Petrogradsky zeta function $Z(R^d;t_1,\dots,t_d)$ as a generating function for the $W$ statistics.

\begin{corollary}\label{cor:refined_sum}
We have
\begin{equation}\label{eq:refined_sum}
\sum_{x\in X} q^{W(x)} t_1^{n(\delta^d(x))} t_2^{n(\delta^{d-1}(x))}\dots t_d^{n(\delta^1(x))} = Z(\Fq[[T]]^d;t_1,\dots,t_d).
\end{equation}
\end{corollary}

Note that $n(\delta^d(x)), \dots, n(\delta^1(x))$ is nothing but sorting $n_1,\dots,n_d$ in descending order, similar to the role of $\lambda_1(M),\dots,\lambda_d(M)$. When specialized at $t_1=\dots=t_d=t$, the left-hand side is equal to $\sum_{x\in X} q^{W(x)} t^{n(x)}$, the subject of Theorem \ref{thm:full_sum_intro}. Notice that in this special case, the denominator given in Theorem \ref{thm:petrogradsky} is not tight: the denominator is $\prod_{j=1}^d \parens*{1-\parens*{q^{d-j}t}^j}$, which is divisible by but not equal to the denominator $\prod_{j=1}^d (1-q^{d-j}t)$ in Theorem \ref{thm:full_sum_intro}.

Our next goal is to compute the left-hand side of \eqref{eq:refined_sum}. This will recover Petrogradsky's theorem, as another application of the spiral shifting operators. We first give a brief overview of several alternative interpretations of the positive-coefficient polynomial $w_D(q)$ given in \cite[\S 5]{petrogradsky2007multiple}.

\subsection{Descents and multiset permutations} 

Given $D\subeq [d-1]$ (recall that we shall view it as a set of descents of some permutation), we introduce some notation. Let $D=\set{d_1<\dots<d_k}$, where $k=\abs{D}$. For $j=0,1,\dots,k$, write $m_j=d_{j+1}-d_j$, where $d_0:=0$ and $d_{k+1}:=d$. 

Cosider a standard $\inv$-preserving bijection frequently used in Schubert calculus:
\begin{equation}\label{eq:multiset_permutation}
\begin{aligned}
\beta_D: \set[\big]{\pi\in S_d: D(\pi)\subeq D} &\to \set[\big]{\sigma\in \mathrm{Map}([d],\set{0,\dots,k}): \abs{\sigma^{-1}(i)}=m_i} \\
\pi &\mapsto (\pi(i)\mapsto\text{unique $j$ such that }d_j<i\leq d_{j+1}).
\end{aligned}
\end{equation}

For example, if $d=5$, $D=\set{2,4}$, $\pi=25|13|4$ (the verticle bars denote positions where descents are allowed; note how they encode $(m_0,m_1,m_2)=(2,2,1)$), then $\beta_D(\pi)$ sends $2,5$ to $0$, and $1,3$ to $1$, and $4$ to $2$. The two-line notation
$
\parens*{
\begin{matrix}
25\\
00
\end{matrix}
\,\vrule\,
\begin{matrix}
13\\
11
\end{matrix}
\,\vrule\,
\begin{matrix}
4\\
2
\end{matrix}
}
$
is instructive. We shall view $\beta(\pi)$ as a permutation of the multiset $0^{m_0}1^{m_1}\dots k^{m_k}$; in this example, $\beta_D(\pi)$ is the multiset permutation $10120$. It is instructive to read $\beta_D(\pi)$ by permutating the columns of the two-line notation above, such that the first row is sorted: $\displaystyle{12345\choose 10120}$. We always have $\inv(\pi)=\inv(\beta_D(\pi))$. In this example, both numbers of inversions are 4.

The map $\beta_D$ induces a bijection
\begin{equation}\label{eq:strict_multiset_permutation}
\set[\big]{\pi\in S_d: D(\pi)=D} \map[\cong] \set[\big]{\sigma: \text{For any $j\in [d]$, there is $i_1<i_2, \sigma(i_1)=j, \sigma(i_2)=j-1$}}.
\end{equation}

We call the latter the set of \textbf{strict} multiset permutations, namely, multiset permutations that has at least a $j$ put before a $j-1$ for any $1\leq j\leq d$. For example, $10120$ is not a strict multiset permutation because every $1$ is before every $2$, and this corresponds to the fact that the second vertible bar of $\pi=25|13|4$ is not actually a descent.

The bijection \eqref{eq:strict_multiset_permutation} gives the alternative formula \cite[Theorem 5.3(1)]{petrogradsky2007multiple}:
\begin{equation}
w_D(q)=\sum_{\sigma\text{ strict}} q^{\inv(\sigma)},
\end{equation}
where the sum is over all strict multiset permutations of $0^{m_0}1^{m_1}\dots k^{m_k}$.

Another formula for $w_D(q)$ can be obtained from a formula of MacMahon \cite[\S 3.4]{andrewspartitions} 
\begin{equation}
\sum_{\pi\in S_d, D(\pi)\subeq D} q^{\inv(\pi)} = {d \brack m_0,m_1,\dots,m_k}_q,
\end{equation}
which can be viewed as the Schubert cell decomposition of the flag variety $\mathrm{Fl}(D;d):=\mathrm{Fl}(d_1,\dots,d_k;d)$. We denote by ${d \brack D}_q$ the Gaussian $q$-multinomial polynomial ${d \brack m_0,m_1,\dots,m_k}_q$. Using the M\"obius inversion formula and the expression of the M\"obius function on the Boolean poset of subsets of $[d-1]$ (see for example, \cite[Proposition 2.44]{orlikterao1992}), one obtains
\begin{equation}\label{def:wDq_petrogradsky}
w_D(q)=\sum_{D'\subeq D} (-1)^{\abs{D}-\abs{D'}} {d \brack D'}_q.
\end{equation}

Petrogradsky \cite[\S 3]{petrogradsky2007multiple} used this instead of \eqref{def:wDq} as the original definition of $w_D(\lambda)$.

\subsection{Alternative proof of Petrogradsky's theorem for $\Fq[[T]]^d$}
We now compute the left-hand side of \eqref{eq:refined_sum}. Let $\Gamma'$ be the subsemigroup of $\Gamma$ (freely) generated by $g_1^d,g_2^{d-1},\dots,g_d^1$. We notice that the effect of $g_{d+1-j}^j$ on $X$ is raising the level of the highest $j$ points of $\Delta(x)$ by $1$ while fixing their seats; this is because $g_{d+1-j}$ is spiral shifting within the highest $j$ elements, so shifting $j$ times completes a full cycle. For $x\in X$, we define the \textbf{seat permutation} of $x$ to be the permutation $i(\delta^1(x)), i(\delta^2(x)), \dots, i(\delta^d(x))$.

\begin{lemma}\label{lem:orbit_decomposition}
There is an orbit decomposition of $X$:
\begin{equation}
X=\bigsqcup_{\pi\in S_d} \Gamma'\cdot x_\pi
\end{equation}
for certain $d!$ elements $x_\pi\in X$. Moreover, each orbit consists precisely of elements of $X$ with seat permutation $\pi$.
\end{lemma}
\begin{proof}
The index of $\Gamma'$ in $\Gamma$ is $d\cdot (d-1) \cdot \dots \cdot 1 = d!$, and $\Gamma$ is clearly a disjoint union of $d!$ cosets of $\Gamma'$. Say $\Gamma=\bigsqcup_{k=1}^{d!} a_k \Gamma'$ for some $a_k\in \Gamma$. Then we have a decomposition of $X$ into $d!$ disjoint $\Gamma'$ orbits:
\begin{equation}
X=\Gamma\cdot \mathbf{0}=\bigsqcup_{k=1}^{d!} \Gamma'\cdot (a_k\cdot \mathbf{0})
\end{equation}

By the discussion above this lemma, each orbit has a constant seat permutation. Let $\pi_k$ be the seat permutation of $a_k\cdot \mathbf{0}$ (and thus the seat permutation of every element in the $k$-th orbit). Every $\pi$ occurs as some $\pi_k$, since there exists an element of any given seat permutation (for instance, if $\pi=i_1 \dots i_d$, then consider the element $x$ with $\Delta(x)=\set{(i_1,1),\dots,(i_d,d)}$). Because there are $d!$ permutations and $d!$ orbits, all $\pi_k$ must be distinct. Hence the proof of the lemma is complete by relabeling. 
\end{proof}

Define
\begin{equation}
\Cont'(x):=q^{W(x)}t_1^{n(\delta^d(x))} t_2^{n(\delta^{d-1}(x))}\dots t_d^{n(\delta^1(x))},
\end{equation}
so that the left-hand side of \eqref{eq:refined_sum} is $\sum_{x\in X} \Cont'(x)$. We have
\begin{align}
\Cont'(g_{d+1-j}^j(x)) &= q^{W(g_{d+1-j}^j(x))} t_1^{n(\delta^d(x))+1} \dots t_j^{n(\delta^{d+1-j}(x))+1} t_{j+1}^{n(\delta^{d-j}(x))} \dots t_d^{n(\delta^1(x))} \\
&= q^{j(d-j)} t_1 t_2 \dots t_j \Cont'(x)\\
&=: z_j \Cont'(x),
\end{align}
where the first equality results from the discussion at the beginning of this section, and the second equality is by Lemma \ref{lem:weight}. The orbit decomposition thus imply the following formula.

\begin{theorem}
We have
\begin{equation}\label{eq:refined_sum_in_content}
\sum_{x\in X} q^{W(x)} t_1^{n(\delta^d(x))} t_2^{n(\delta^{d-1}(x))}\dots t_d^{n(\delta^1(x))}
=
\frac{1}{(1-z_1)\dots (1-z_d)}\sum_{\pi\in S_d} \Cont'(x_\pi).
\end{equation}
\end{theorem}
\begin{proof}
Perform the same proof as Theorem \ref{thm:free_semigroup_orbit} for each orbit in Lemma \ref{lem:orbit_decomposition}.
\end{proof}

In particular, we have proved that $Z(\Fq[[T]]^d;t_1,\dots,t_d)$ is rational with the desired denominator. Moreover, we obtain yet another formula $\sum_{\pi\in S_d} \Cont'(x_\pi)$ for the numerator, which implies that the numerator is a positive-coefficient polynomial in $q,t_1,\dots,t_d$ with total coefficient $d!$ (i.e., evaluates to $d!$ at $q=t_1=\dots=t_d=1$). 

We now verify that \eqref{eq:refined_sum_in_content} does recover Petrogradsky's formula \eqref{eq:petrogradsky}. This means that we shall prove
\begin{equation}
\text{Claim: }\sum_{\pi\in S_d} \Cont'(x_\pi) = \sum_{D\in [d-1]} \parens*{w_D(q^{-1})\prod_{j\in D} z_j}.
\end{equation}

We first observe that $x_\pi$ is given by the strict multiset permutation $\beta_{D(\pi)}(\pi)$. For example, if $d=5$ and $\pi=25|14|3$ (labeling all descents with verticle bars), then $\beta_{D(\pi)}(\pi)=10210$ and $x_\pi=(1,0,2,1,0)$. Indeed, from the orbit decomposition in Lemma \ref{lem:orbit_decomposition}, the element $x_\pi$ must be the ``lowest possible'' element with seat permutation $\pi$. If we write $\pi=i_1i_2\dots i_d$ and $\Delta(x_\pi)=\set{(i_1,n_{i_1})\prec (i_2,n_{i_2}) \prec \dots \prec (i_d,n_{i_d})}$, then $x_\pi$ is obtained from greedily assigning the lowest possible values for $n_{i_1}, n_{i_2}, \dots, n_{i_d}$ such that $(i_1,n_{i_1})\prec (i_2,n_{i_2}) \prec \dots \prec (i_d,n_{i_d})$. The resulting $x_\pi$ is precisely $\beta_{D(\pi)}(\pi)$, because we start at $n_{i_1}=0$ and the level of the next point must rise by $1$ if the seat descends. 

We then claim that 
\begin{equation}
t_1^{n(\delta^d(x_\pi))} t_2^{n(\delta^{d-1}(x_\pi))}\dots t_d^{n(\delta^1(x_\pi))} = \prod_{j\in D} y_{d-j},
\end{equation}
where $y_j:=t_1t_2\dots t_j$. The proof is clearer if illustrated on an example. Consider $\pi=16|235|48|7$, then $D(\pi)=\set{2,5,7}$ and $(m_0,\dots,m_3)=(2,3,2,1)$. We have $\Delta(x_\pi)=\parens*{
\begin{matrix}
16\\
00
\end{matrix}
\,\vrule\,
\begin{matrix}
235\\
111
\end{matrix}
\,\vrule\,
\begin{matrix}
48\\
22
\end{matrix}
\,\vrule\,
\begin{matrix}
7\\
3
\end{matrix}
}$, which reads $\set{(1,0),(6,0),\dots, (7,3)}$. Then reading the levels of $\Delta(x_\pi)$ from high to low, we obtain
\begin{equation}
t_1^{n(\delta^d(x_\pi))} t_2^{n(\delta^{d-1}(x_\pi))}\dots t_d^{n(\delta^1(x_\pi))}=t_1^3 t_2^2 t_3^2 t_4^1 t_5^1 t_6^1 t_7^0 t_8^0 = y_1 y_3 y_6 = \prod_{j\in D} y_{d-j}.
\end{equation}

We now compute $W(x_\pi)$. Write $\pi=i_1i_2\dots i_d$, $D(\pi)=\set{d_1<\dots<d_k}$, and
\begin{equation}\label{eq:two_line_strict_multiset_permutation}
\Delta(x_\pi) = 
\parens*{
\begin{matrix}
i_1 & \dots & i_{d_1} \\
0   & \dots & 0
\end{matrix}
\, \vrule \,
\begin{matrix}
i_{d_1+1} & \dots & i_{d_2} \\
1   & \dots & 1
\end{matrix}
\, \vrule \, \cdots \, \vrule \,
\begin{matrix}
i_{d_k+1} & \dots & i_{d} \\
k   & \dots & k
\end{matrix}.
}
\end{equation}

Note that the sequence $n(\delta^1(x_\pi)), \dots, n(\delta^d(x_\pi))$ is the second row of \eqref{eq:two_line_strict_multiset_permutation}, which only depends on $D=D(\pi)$. We denote the sequence by $n_1(D),\dots,n_d(D)$; it is $m_0$ copies of $0$'s, followed by $m_1$ copies of $1$'s, ..., $m_k$ copies of $k$'s. Then for $a<b$, the columns $\delta^a(x_\pi), \delta^b(x_\pi)$ have distance $n_b(D)-n_a(D)$ if $i_a < i_b$, or the above number minus one if $i_a>i_b$. Hence
\begin{equation}
W(x_\pi) = -\inv(\pi)+ \sum_{1\leq a<b\leq d} (n_b(D)-n_a(D)).
\end{equation}

We shall simplify $\sum_{a<b} (n_b(D)-n_a(D))$. For $0\leq r<s\leq k$, the difference $s-r$ occurs as the difference $n_b(D)-n_a(D)$ for $m_{r}m_{s}$ times. Thus
\begin{equation}
\sum_{a<b} (n_b(D)-n_a(D))=\sum_{0\leq r<s\leq k} m_r m_s (s-r).
\end{equation}

One may verify that
\begin{equation}
\sum_{0\leq r<s\leq k} m_r m_s (s-r) = m_0(m_1+\dots+m_k) + (m_0+m_1)(m_2+\dots+m_k)+\dots+(m_0+\dots+m_{k-1}) m_k
\end{equation}
by noting that the term $m_r m_s$ occurs $s-r$ times on the right-hand side. Recalling that $m_0,\dots,m_k$ are the gaps between consecutive terms of $0,d_1,\dots,d_k,d$, we get
\begin{equation}
\sum_{a<b} (n_b(D)-n_a(D))=\sum_{j\in D} j(d-j).
\end{equation}

Putting the above together, we get
\begin{equation}
\Cont'(x_\pi)=q^{-\inv(\pi)} q^{\sum_{j\in D(\pi)} j(d-j)} \prod_{j\in D(\pi)} y_{d-j} = q^{-\inv(\pi)} \prod_{j\in D(\pi)} z_{d-j}.
\end{equation}

In particular, we obtain a formula for the numerator of Petrogradsky's zeta function as a summation over all partitions of $S_d$ that keeps track of the number of inversions and the set of descents:
\begin{equation}\label{eq:refined_sum_permutation_flipped}
\sum_{x\in X} q^{W(x)} t_1^{n(\delta^d(x))} t_2^{n(\delta^{d-1}(x))}\dots t_d^{n(\delta^1(x))}
=
\frac{1}{(1-z_1)\dots (1-z_d)}\sum_{\pi\in S_d} \parens*{q^{-\inv(\pi)} \prod_{j\in D(\pi)} z_{d-j}}.
\end{equation}

Collecting all $\pi$ with the same $D(\pi)$, we get
\begin{equation}\label{eq:petrogradsky_flipped}
\sum_{x\in X} q^{W(x)} t_1^{n(\delta^d(x))} t_2^{n(\delta^{d-1}(x))}\dots t_d^{n(\delta^1(x))}
=\frac{1}{(1-z_1)\dots (1-z_d)} \sum_{D\subeq [d-1]} \parens*{w_D(q^{-1}) \prod_{j\in D} z_{d-j}}.
\end{equation}

The slight difference between the above formula and \eqref{eq:petrogradsky} can be addressed by considering the involution on the power set of $[d-1]$ defined by $D\mapsto D^t:=\set{d-j:j\in D}$, and the involution on $S_d$ that sends $\pi=i_1 \dots i_d$ to its ``retrograde inversion'' $\pi^t:= (d+1-i_d) (d+1-i_{d-1}) \dots (d+1-i_1)$.\footnote{``Retrograde inversion'' is a musical term that describes such a transformation on a melody. Equivalently, $\pi^t=w_0 \pi w_0$, where $w_0=d(d-1)\dots 1$ is the longest word of $S_d$.} We have $\inv(\pi^t)=\inv(\pi)$ and $D(\pi^t)=D(\pi)^t$, so that $w_D(q)=w_{D^t}(q)$. By a change of variable $D\to D^t$, the formula \eqref{eq:petrogradsky_flipped} implies \eqref{eq:petrogradsky}, and the alternative proof of Petrogradsky's theorem is complete. As a takeaway of the above analysis, we rewrite \eqref{eq:refined_sum_permutation_flipped} in a slightly better form, obtained from \eqref{eq:refined_sum_permutation_flipped} with $\pi\mapsto \pi^t$, and state some functional equations as direct consequences of this form.

\begin{corollary}[\cite{petrogradsky2007multiple}]
We have
\begin{equation}\label{eq:refined_sum_permutation}
Z(\Fq[[T]]^d;t_1,\dots,t_d)=\frac{\sum_{\pi\in S_d} \parens*{q^{-\inv(\pi)} \prod_{j\in D(\pi)} z_{j}}}{\prod_{j=1}^d (1-z_j)},
\end{equation}
where $z_j=q^{j(d-j)}t_1t_2\dots t_j$. Moreover, the numerator is a polynomial in $q^{-1},z_1,\dots,z_{d-1}$ that is invariant under the simultaneous change of variable $z_j\mapsto z_{d-j}$. In particular, $Z(\Fq[[T]]^d;t_1,\dots,t_d)$ is a rational function $Z(q;z_1,\dots,z_d)$ in $q,z_1,\dots,z_d$ that satisfies functional equations
\begin{enumerate}
\item $Z(q;z_{d-1},\dots,z_1, z_d)=Z(q;z_1,\dots,z_{d-1},z_d)$;
\item $Z(q^{-1};z^{-1},\dots,z_d^{-1})=(-1)^d q^{{d\choose 2}} Z(q;z_1,\dots,z_d)$, or equivalently, 
\begin{equation}\label{eq:function_equation}
Z(\Fq[[T]]^d;t_1,\dots,t_d)|_{q\mapsto q^{-1}, t_i\mapsto t_i^{-1}}=(-1)^d q^{{d\choose 2}} Z(\Fq[[T]]^d;t_1,\dots,t_d).
\end{equation}
\end{enumerate}

\end{corollary}

\begin{proof}
The assertion (a) is obvious from the previous discussions. For the assertion (b), notice that the permutation $\pi'=w_0\pi$ where $w_0=d(d-1)\dots 1$ is the longest word satisfies $\inv(\pi')={d \choose 2}-\inv(\pi)$ and $D(\pi')=[d-1]\setminus D(\pi)$. Thus
\begin{equation}
\begin{aligned}
Z(q^{-1};z_1^{-1},\dots,z_d^{-1})&= \frac{\sum_{\pi\in S_d} \parens*{q^{\inv(\pi')} \prod_{j\in D(\pi')} z_{j}^{-1}}}{\prod_{j=1}^d (1-z_j^{-1})} \\
&=\frac{(-1)^d z_1\dots z_d}{\prod_{j=1}^d (1-z_j)} \cdot \sum_{\pi\in S_d} \parens*{q^{{d \choose 2}-\inv(\pi)}  \cdot z_1^{-1}\dots z_d^{-1} \prod_{j\in D(\pi)} z_{j}}\\
&=(-1)^d q^{{d\choose 2}} Z(q;z_1,\dots,z_d).
\end{aligned}
\end{equation}

The final equivalent statement is because as $q,t_1,\dots,t_j\mapsto q^{-1},t_1^{-1},\dots,t_d^{-1}$, we have $z_j\mapsto z_j^{-1}$. 
\end{proof}

\begin{remark}\label{rmk}
As a final remark, we compare our proof to other known proofs of \eqref{eq:refined_sum_permutation}.  Petrogradsky's proof builds upon results that for each abelian $p$-group $G$, count subgroups $H\subeq G$ of a given cotype (i.e., the isomorphism class). His proof arrives at a formula in terms of Gaussian $q$-multinomial polynomials (\eqref{eq:petrogradsky} but with $w_D(q)$ defined by \eqref{def:wDq_petrogradsky}), which turns out somewhat miraculously to be interpretable as a permutation sum, and hence the numerator of \eqref{eq:petrogradsky} has nonnegative coefficients. The work of du Sautoy and Lubotzky \cite[Theorem 5.9]{dusautoylubotzky} recovers \eqref{eq:refined_sum_permutation} as a special case where $G=\GL_d$ and $\rho$ is the standard representation. Their proof uses $p$-adic integral over $\GL_d(\Zp)$ and the Iwahori decomposition, and reaches the permutation sum (or more generally, a sum over the (affine) Weyl group) and the functional equation \eqref{eq:function_equation}. See also the concluding remark of \cite[\S 2]{ckk2017cotype}. Our proof is based on a matrix normal form \eqref{eq:hlex_normal_form}, and it leads to the permutation sum of \eqref{eq:refined_sum_permutation} via the orbit decomposition in Lemma \ref{lem:orbit_decomposition}. 
\end{remark}

\section*{Acknowledgements}
We thank Jason Bell and Nathan Kaplan for informing the authors about \cite{ito1969semilinear, petrogradsky2007multiple}, respectively. We thank helpful conversations with Farid Aliniaeifard, Jim Bryan, Gautam Chinta, Daniel Erman, Asvin G, Jeffery Lagarias, Mark Shimozono, Yifan Wei, Stephanie van Willigenburg and Michael Zieve. The first author thanks the AMS-Simons Travel Grant.

\end{document}